\documentclass[11pt]{amsart}

\usepackage{amsmath}
\usepackage{amssymb}
\usepackage{amsthm}
\usepackage{amscd, amsfonts}
\usepackage[dvips]{epsfig}
\usepackage{verbatim}
\usepackage{epsf}

\usepackage{mathrsfs}

\newcommand{\bj}{\bar{j}}

\newcommand{\al}{\alpha}
\newcommand{\bal}{\bar{\alpha}}
\newcommand{\be}{\beta}
\newcommand{\bbe}{\bar{\beta}}
\newcommand{\ga}{\gamma}

\newcommand{\Ga}{\Gamma}
\newcommand{\De}{\Delta}
\newcommand{\de}{\delta}

\newcommand{\eps}{\epsilon}
\newcommand{\ka}{\kappa}

\newcommand{\m}{\mu}

\newcommand{\n}{\nu}
\newcommand{\bn}{\bar{\nu}}

\newcommand{\tta}{\theta}

\newcommand{\Si}{\Sigma}

\newcommand{\scrD}{\mathscr{D}}

\newcommand{\fz}{\mathfrak{z}}

\newcommand{\calG}{\mathcal{G}}
\newcommand{\calH}{\mathcal{H}}

\newcommand{\calK}{\mathcal{K}}

\newcommand{\om}{\omega}
\newcommand{\Om}{\Omega}

\newcommand{\bbC}{\mathbb{C}}

\newcommand{\bbN}{\mathbb{N}}

\newcommand{\del}{\operatorname{\partial}}
\newcommand{\bdel}{\operatorname{\bar{\partial}}}

\newcommand{\wge}{\wedge}

\theoremstyle{plain}
\newtheorem{theorem}{Theorem}[section]
\newtheorem{lemma}[theorem]{Lemma}
\newtheorem{proposition}[theorem]{Proposition}
\newtheorem{corollary}[theorem]{Corollary}

\theoremstyle{definition}

\newtheorem{definition}[theorem]{Definition}

\numberwithin{equation}{section}

\author[Aleyasin]{S. Ali Aleyasin}
\address{Department of Mathematics\\
Stony Brook University\\ Stony Brook, NY 11794}
\email{aleyasin@math.sunysb.edu}

\thanks{I am grateful to my adviser, Prof. Xiu-Xiong Chen, for introducing me to this circle of problems and for his constant encouragement. I am also grateful to Prof. Eric Bedford, Prof. Simon Donaldson,  and to Prof. Claude LeBrun for generously sparing  time and for fruitful discussions, and to Kai Zheng  for carefully reading the draft of the first part of the note and for making useful suggestions.
}

\title[ the space of ALE K\"ahler metrics]{ the space of K\"ahler potentials on 
an Asymptotically locally euclidean K\"ahler manifold }

\begin{document}

\maketitle

\begin{abstract}
In this note, we shall prove geodesic convexity of the space of K\"ahler potentials on an ALE K\"ahler manifold.
This extends earlier results in the compact case proved in the fundamental work of X-X. Chen.
We further prove the boundedness from below of the Mabuchi energy, and give an  prove
the uniqueness of scalar-flat metrics in this category when \(c_1 \leq 0\).
53C55, 
35J60 

\end{abstract}

\section{introduction}
In this note, we prove a generalisation of a the result of  X-X. Chen reported in \cite{Ch00} for compact K\"ahler manifolds to the case of asymptotically locally euclidean K\"ahler manifolds.

After the work of Bourguignon, Donaldson, Mabuchi, and Semmes, it is now well-known that the geodesics in the space of K\"ahler metrics have deep connections to the questions of (non)-existence and uniqueness of extremal K\"ahler metrics and metrics of constant scalar curvature.
In a fundamental work, 
X-X. Chen established the existence of geodesics in the space of K\"ahler potentials with bounded weak \(dd^c\)-derivatives (also referred to as weak \(C^{1,1}\) solutions), and by-passed the lack of existence of higher derivatives for proving the uniqueness of constant-scalar-curvature and extremal metrics.
Since the estimates in \cite{Ch00} require strict positivity of the boundary conditions, in a more recent work, W. He in \cite{He12} studied the existence of geodesics  with possible degeneracies on the end points.
There it is proved that even with degenerate boundary conditions, 
the solution will have weak bounded laplacian in the space direction.
Also, by virtue of an observation made by Berndtsson, 
the first time derivative is also bounded, see \cite{Be11}.
Making use of  the calculations by W. He in \cite{He12} in an essential way, 
X-X. Chen and the author extended this to singular potentials to derive weighted estimates of the laplacian close to the singularity \cite{ACh13}.

In the current work, we prove the existence of weak solutions with bounded \(dd^c\)-derivatives and derive decay estimates for the potential and time derivatives. This in particular implies that on each time slice the metric is an ALE metric in the extended sense.
Further, we show the uniqueness of ALE metrics of constant scalar curvature in each K\"ahler class.
If we view geodesics as curves with vanishing acceleration, as we shall see, the main r\^{o}le will be played by certain curves with preassigned non-zero acceleration which we will refer to as the \(\eps\)-geodesics.

\begin{theorem}
\label{thm-1-1}
Let \(M\) be an asymptotically locally euclidean K\"ahler manifold.
Assume that \(\phi_0\) and \(\phi_1\) are two potentials belonging to \(\calH_{ALE}\).
Then, there is a unique geodesic with spatial laplacian, \(\De\) satisfying the decay property:

\begin{eqnarray}
\label{decay-1}
	\vert \De \phi  \vert \leq C, \phi = O(r^{-2n + 2}), \del_t \phi = O(r^{-2n + 2}), \vert \del_{tt} \phi \vert \leq C
\end{eqnarray}

wherein \(C\) depends only on the end points and on the lower bound on the curvature of the reference metric \(\om\).
In particular, at each time slice, the potential satisfies 
\[
	\phi(t) = \tilde{\calH}_{ALE}
\]

\end{theorem}

One could derive the Euler-Lagrange equation associated to the energy of curves on \(\calH_{ALE}\) to be the following:
\[
	\calG(\phi):= \phi'' - g^{\al \bbe}_\phi \phi'_\al \phi'_{\bbe} = 0
\]
wherein the prime sign denotes time derivative.
As in the compact, we interpret the equation in the way that the problem is reduced to solving a degenerate complex Monge-Amp\`ere equation.
Namely, let us define \(\Si = [0,1] \times S^1\), and view it as a Riemann surface with boundary.
Extend the potentials on the \(S^1\) factor in the trivial way.
Further, let \(\pi: M \times \Si \to M\) be the obvious projection.
By pulling back the metric \(\om\), we shall obtain \(\Om:= \pi^* \om\).
Then, on \(M \times \Si\) we may consider \(\Om\)-plurisubharmonic potentials.
One may then see by a calculation  that 
\begin{equation}
\label{11}
\Om_\phi^n = \calG(\phi) {\om_\phi^n \over \om^n}
\end{equation}
 which allows us to solve the boundary value problem \eqref{eq-01} instead.
 See \cite{Do96} for more on this construction.
Notice however that since there is not yet know a lower bound on the rank of the complex hessian, \(\phi_{;\m \bn}\) of the solutions to \eqref{eq-01}, one cannot
guarantee the non-degeneracy of the volume form \(\om_\phi^n\).
As a result,
satisfying \eqref{eq-01}, although a necessary condition, is not
sufficient for \(\calG(\phi)=0\) to hold.
One may therefore think of \eqref{eq-01} as generalised geodesics.

The proof of the Theorem \ref{thm-1-1} 
is based on the resolution of the geodesic equation and proving appropriate asymptotic behaviour as the following theorem states:

This theorem will be a corollary to the next theorem, asserting the existence of a weak solution to the geodesic equation.

\begin{equation}
\label{eq-01}
\begin{cases}
 \Om_\phi^{n+1} =
0\\
\phi(x,i) = \phi(i), \text{  } i = 0, 1. 
\end{cases}
\end{equation}

\begin{theorem}
\label{thm-1-2}
Assume that the boundary conditions in the boundary value problem \eqref{eq-01} belong to \(\calH_{ALE}\).
Then, there exists a weak solution in the sense that it is continuous with bounded weak derivative satisfying the decay rates \eqref{decay-1}.
\end{theorem}

\textit{Proof of Theorem \ref{thm-1-2}}
For the proof, we shall approximate the zero right hand side by strictly positive ones that tend to zero and derive estimates independent of the lower bound of the right hand side, \(f\), that will guarantee the existence of a weak solution by the Arzel\`a-Ascoli theorem. 
This is done in the following sections.
In Section \ref{section-3}, we construct classical solutions for positive right hand side, \(f\), on the strip.
In Section \ref{section-4}, we derive weighted estimates independent of the lower bound of the right hand side, and thereby guarantee the decay rate of the laplacian of the weak solutions. 
This will prove that the same bounds hold weakly once one passes to the uniform limit obtained by applying the Arzel\`a-Ascoli theorem on compact subsets of the strip.


\begin{theorem}
\label{uniqueness}
Let \(M^n\) be an ALE K\"ahler space with \(c_1(M) \leq 0\).
Then, there is at most asymptotically locally euclidean K\"ahler metric of constant scalar curvature in each cohomology class.
In the particular case when \(c_1 = 0\), 
in each K\"ahler class there exists one and only one scalar-flat metric which is further Ricci-flat.
\end{theorem}

Besides the uniqueness issue, we can further prove the boundedness from below of  \emph{Mabuchi's \(\calK\)-energy} as asserted in the following:

\begin{theorem}
\label{boundedness}
Let \(M\) be an asymptotically locally euclidean K\"ahler manifold.
Then, in each cohomology class, 
the metric of constant scalar curvature realises the global minimum of the \(\calK\)-energy. 
\end{theorem}

In the case of vanishing first Chern class, `Scalar-flat ALE K\"ahler metrics are Ricci-flat'. 
This assertion can already be proved using more standard methods as we shall describe in 
\textsection \ref{mabuchi}.
For the existence of Ricci-flat metrics in the case of vanishing \(c_1\) we rely on  the work of Joyce on the extension of the Calabi conjecture to the ALE K\"ahler spaces. 
Along with Theorem \ref{boundedness}, when \(c_1(M)=0\), the \(\calK\)-energy is bounded from below and there  always exists a metric of zero scalar curvature in each class which realises the minimum.

In the case of \(c_1(M) < 0\) however, the uniqueness result does not seem to follow from the methods known before.

\section{Notation and definitions}

In this section, we introduce the basic notations and definitions.
The reader can find extensive background material for the subject in \textsection 8 of Joyce's book \cite{Jo00}. 

  In what follows, we shall always consider operators such as laplacian and intrinsic derivatives in terms of the reference smooth ALE K\"ahler metric; the same is the case for constants in the estimates whose dependence is not explicitly stated. Recall that an an asymptotically locally euclidean, abbreviated to \emph{ALE}, is a riemannian manifold  that resembles \(\bbC^n/G\) at distant points. 

To make this idea more specific, let us fix a finite subgroup of \(G \subset SU(n)\) that acts freely on \(\bbC^n - \{0\}\).
Then, the euclidean metric \(h_0\) on \(\bbC^n\) descends to the quotient \(\bbC^n - \{ 0 \} / G\).
Let \(r\) be the euclidean distance on \(\bbC^n\).
Then, we have the following definition:

\begin{definition}
Let \((M^n, J, g)\), which henceforth  we shall denote by \(M\) for the sake of brevity, be a non-compact K\"ahler manifold of dimension \(n\).
We say that \(M^n\) is \emph{asymptotically locally euclidean} asymptotic to \(\bbC^n / G\) provided that there exists a compact set \(S \subset \subset M\) and a map \(M - S \stackrel{\pi^{-1}}{\to} \bbC^n / G\) that is a diffeomorphism between \(X - S\) and the set \( \bbC^n - B_0(R) \) for some fixed \(R\).
We require the metric \(g\) to satisfy 
\[
 \nabla^k (\pi_*(g) - h_0 ) = O(r^{-2n - k})  \textrm{  for  } k \geq 0
\]
wherein \(\nabla\) is the Levi-Civita connexion associated to the flat metric \(h_0\).

\end{definition}

Along the same lines, in order to parametrise the space of metrics, let us introduce the space of ALE K\"ahler potentials. 
Unlike the case of compact manifolds, 
there is no ambiguity of adding a constant and to each K\"ahler metric in the K\"ahler class there corresponds only one potential.
We have the following definition:

\begin{definition}
\label{ale-def-2}
For a given asymptotically locally euclidean K\"ahler manifold \((M, \om, J)\) we define the space of ALE K\"ahler potentials to be as follows:

\[
 	\calH_{ALE} := \{ \phi \in C^\infty |   \om + dd^c \phi > 0, \nabla^k \phi = O(r^{2 - 2n - k}) , 0 \leq k \leq 2 \}
\]

Also, we can define a weaker space to which we may refer as the \emph{zero-th order} ALE K\"ahler potentials:

\[
	\tilde{\calH}_{ALE} =  \{ \phi |   \om + dd^c \phi \geq 0, \phi = O(r^{2-2n}),  
\vert \De \phi \vert \leq C \}
\]

\end{definition}

In particular, elements of \(\tilde{\calH}_{ALE}\) give rise to bounded metrics.

In the rest of this note, we shall refer to the laplacian operators of the metrics \(\om\) and \(\om_\phi\) on each time slice by \(\De\) and \(\De_\phi\).
In order to denote the laplacian on the total space \(M \times \Si\) with respect to the K\"ahler forms \(\Om\) and \(\Om_\phi\) we shall use \(\tilde{\De}\) and \(\tilde{\De}_\phi\).


We also define the following weighted version of H\"older spaces. 
For some negative real number \(\be\), let \(\Vert f\Vert_{C^k_\be}\) be defined as
\begin{eqnarray}
	\Vert f \Vert_{C^k_\be}:= \sum_{j=1}^k \sup_M \left \vert r^{j-\be} \nabla^j f\right \vert \nonumber
\end{eqnarray}
Let \(\de\) be the injectivity radius of the metric \(\om_0\), and let \(d(x,y)\) denote the distance between \(x\) and \(y\) with respect to \(\om_0\). 
Since the definition is supposed to take farther points into account, 
one may as well think of the euclidean distance pushed forward  via the chart \(\pi: \bbC^n - B(0,R) \to M - K\).
Also, let the semi-norm  \([.]_{\al, \ga}\) be defined as follows:
\begin{eqnarray}
	\left [ f \right]_{\al, \ga} \sup_{\stackrel{x \neq y}{d(x,y) < \delta}} \left( \left( (r(x) \vee r(y) \right)^{-\ga} { \vert f(x) - f(y) \vert \over d(x,y)^\al} \right)
\end{eqnarray}

wherein \(\vee\) denotes the minimum of two numbers. 
The definitions extend from functions to tensors in the obvious way. 
We then define the space \(C^k_\be\) to consist of functions that have finite \(\Vert . \Vert_{C^{k,\al}_\be} \)-norm defined as follows:
\begin{eqnarray}
	\Vert f \Vert_{C^{k,\al}_\be} := \Vert f \Vert_{C^k_\be} + \left[ \nabla^k f \right]_{\al, \be-k-\al}
\end{eqnarray}

It is probably the appropriate juncture to clarify the meaning of two key notions we shall use in this context: the `first Chern class' and its sign.
In general, notions such as Chern classes do not directly carry over from the framework of compact manifold without boundary to the non-compact case.
Heuristically speaking, we want a notion of the first Chern class that is compatible with the trivial topology of the ALE manifolds outside of some compact set \(K \subset \subset M\).
Therefore, we define an \emph{admissible} hermitian metric \(h\) on \(-K_M\) as follows. In order to make sense of the asymptotic flatness of the hermitian metric \(h\) on the anti-canonical bundle, \(-K_M\), let us use the coordinate system \(\pi\) on \(M - K\).
This trivialisation induces a metric on the bundle \(-K_M\) over \(M - K\), 
the flat metric on \(-K_{\bbC^n - B(0,R)}\), which we call \(h_0\).
We can extend \(h_0\) to the entire manifold \(M\) in a smooth way.
We know that that for any other hermitian metric \(h\) on \(-K_M\) we have \({h \over h_0} = f\) for some function \(f\).
It is therefore enough to demand that \(f\) decays at a certain rate. 
Namely, we can now define:

  be a metric whose curvature, \(\rho_h\), satisfies the decay property
\begin{eqnarray}
	\label{line-bundle-decay}
 	\nabla^k f = O(r^{-2m -k})
\end{eqnarray}

In particular, 
\[
	\vert \rho_h \vert = O(r^{-2m -2}),
\]

We have chosen this decay rate since it admits with the decay rate of the Ricci curvature of ALE metrics, and when we impose such decay rates for the curvature of a line bundle, 
loosely speaking, 
we make the manifold behave like a compact manifold with preassigned behaviour close to the boundary.
In particular, the notion of positivity and negativity for a line bundle can be carried over from the compact case and such notion stays well-define.
More precisely:
\begin{definition}
\label{c-1-bundle}
Let \(L\) be a line bundle over \(M\), an ALE K\"ahler manifold.
We say that \(L\) is a negative (respectively positive) line bundle provided that
there exists an hermitian metric \(h\) on \(L\), which satisfies the decay condition
\eqref{line-bundle-decay}
and further, its curvature form \(\rho_h\) is everywhere a non-positive  (respectively non-negative) (1,1)-form and negative (respectively positive) at some point.
The notion of zero \(c_1\) can be also extended in the same manner.

An ALE K\"ahler  manifold \(M\) is said  to be of negative  positive first Chern class provided that its anti-canonical bundle,  \(-K_M\), is negative or positive respectively.

\end{definition}

We now show that such a notion of \emph{sign} for a line bundle is well-defined.
Let \(\eta_1, \eta_2\) be two closed cohomologous forms with decay rates as in
\eqref{line-bundle-decay},
such that \(\eta_2 \geq 0\) whereas \(\eta_1 < 0\).
Since \(\eta_1\) and \(\eta_2\) are required to satisfy the decay conditions and \([\eta_1 - \eta_2]=0\),
the weighted \(dd^c\)-lemma, Theorem 8.4.4 in \cite{Jo00}, then states that  \(\eta_2 = \eta_1 + dd^c v\) where \(v \in C^{2,\al}_{\be+2}\).
 Gaffney's extension of Stokes's theorem allows us to  integrate by parts and thus observe  that \(\int_M dd^c v \wge \om^{n-1} = 0\). 
 We obtain therefore that
\[
0 \leq \int_M \eta_2 \wge \om^{n-1}= \int_M (\eta_1 + dd^c v) \wge \om^{n-1} 
= \int_M \eta_1 \wge \om^{n-1} 
< 0
\]

which is a contradiction.\\

\section{Classical solution of the equation on the the product of the manifold and the compact Riemann surface with positive right hand side}
\label{section-3}
In this section, we consider the complex Monge-Amp\`ere equation on the product of the asymptotically locally euclidean manifold \(M\) and the cylinder \(\Si\) -viewed as a Riemann surface. 
An appropriately chosen sequence of such solutions will then be used to construct a weak solution to the degenerate equation. 
But the classical solution is important in its own right as we shall see in
\textsection \ref{mabuchi}
as the \(\eps\)-geodesics are our tool in proving Theorems \ref{thm-1-2} and \ref{uniqueness}.
The complex Monge-Amp\`ere equation was solved in \cite{Jo00} on ALE manifolds without boundary, but in our case, the presence of the boundary requires a different treatment.

In order to solve the equation with the right hand side \(f\) asymptotically equal to  a constant, we shall take a sequence of compact domains that expand to the strip. 
In order to prove the existence of classical solutions on the strip, we shall  establish uniform estimates up to order \(C^{2,\ga}\) on compact sets.
We  will prove uniform laplacian and \(L^\infty\) bounds for such solutions.
Existence of the laplacian bounds leads to the uniform ellipticity of the linearised operator which will be used in deriving the estimates for the degenerate case.

\begin{theorem}
Consider the boundary value problem 

\begin{eqnarray}
\label{ma-f-1}
	\begin{cases}
		\Om_\phi^m = e^f \Om^m \text{ ; } M \times \Si \\
		\phi = \psi  \textrm{  ; } \del \left( M \times \Si \right) 
	\end{cases}
\end{eqnarray}
wherein \(M\) and \(\Si\) are an ALE K\"ahler manifold and the cylinder respectively, and \(f \in C^3({M \times \Si})\) satisfies \(f =c \), for some positive number \(c\), outside of some set of the form \(K \times \Si\), where \(K \subset \subset M\).
Then, this problem has a unique solution in \( C^{2, \ga}(M \times \overline{\Si}) \) for some \(\ga\).

\end{theorem}
We observe that the estimates we derive for the compact domains are independent of their size and are therefore uniform.

\textit{Proof of Theorem \ref{ma-f-1}}
As mentioned before, we solve the equation on a sequence of compact domains that grow larger and cover the entire strip. Note that the right hand side is kept constant in the process. The boundary condition has be chosen appropriately as we  shall explain below.
We shall first the detail the technical points that need to be taken into account in the construction of such domain below.
Afterwards, we prove the existence of uniform \(L^\infty\), and laplacian bounds independent of the size of the sets in this family.
In order to make the proof easier to follow, the details of proofs of \emph{a priori} estimates are postponed to  Propositions \ref{l-infty-f-1}.
Having proved the laplacian estimates, one obtains \(C^{2,\ga}\) uniform estimates via the extension of the Evans-Krylov theory to complex hessian equations as is done in \cite{Siu87} for the interior estimates.
The boundary \(C^{2,\ga}\) estimates follow from 
the boundary estimates in the proof Theorem 1
  in \cite{CKN}.
Note that since we obtain a uniform bound on the laplacian,
the equation becomes uniformly elliptic with uniformly bounded complex hessian.
This means that the exponent \(\ga\) and the \(C^{2,\ga}\) norm are uniformly bounded from above on the entire domain \(M \times \Si\).

The construction of the compact domains converging to the strip is as follows.
Let \(B_T\) be the metric ball with respect to the metric \(\om\) on the manifold \(M\). 
Set \(\calG_T \subset M \times \Si \) be the domain obtained by smoothing the corners of the
 region \(B_T \times \Si\).
Let \(\psi_T\) be the function obtained by restricting the function \(\psi\),
constructed in \textsection 
\ref{l-infty},
to \(\calG_T\). \\
 
We solve the problem on each \(\calG_j\) for \(j \in \bbN\) along with uniform estimates up to \(C^{2,\ga}\), i.e
for each \(\calG_j\) we solve 
\begin{eqnarray}
\label{eq-3-2}
	\begin{cases}
		\Om_{\phi_T}^m = e^f \Om^m \textrm{ ; } \calG_T \\
		\phi_T = \psi_T \textrm{  ;  } \del \calG_T
	\end{cases}
\end{eqnarray}

Since the sets \(\calG_j\) exhaust the strip, by uniform continuity on the compact sets and by the usual diagonal argument one may obtain a solution saisfying the same \(C^{2,\ga}\) estimates on the entire strip. \\
We then have to guarantee that these estimates remain valid as \(T \to \infty\).


\begin{proposition}
\label{l-infty-f-1}
In the family of boundary value problems  
\ref{eq-3-2},
we have that for all domains \(\calG_T\), defined in the proof of Theorem \ref{ma-f-1},
the quantities \(\Vert \phi \Vert_{L^\infty}\), \(\Vert \nabla \phi \Vert_{L^\infty} \), and \(\Vert \De \phi \Vert_{L^\infty}\)  are bounded independent of \(T\).
\end{proposition}

\begin{proof}
In order to prove the upper bounds, notice that owing to the fact the boundary data are extended trivially along the \(S^1\)-factor, the solution is indeed convexi in the time direction. 
The convexity in the temporal direction we have that the upper bound in the interior does not exceed that of the boundary.
As one may check, for sufficiently large \(C\), the function \(\psi\), which indeed agrees with the solution on the boundary of the domain \(\calG_T\), clearly serves as a sub-solution and hence provides a lower bound.

In order to prove the laplacian estimates we follow the calculation of Aubin as done in \cite{Siu87}, \textsection 3 of Chapter 2.
Namely, in the normal coordinates at any point one has:

\begin{equation}
	\sum_{j=1}^{n+1} {1 \over 1 + \phi_{j \bj}} - B \leq \tilde{\De}_\phi \left( (n + 1 + \tilde{\De} \phi) + (B + 1) \phi  \right)
\end{equation}

wherein \(B\) is a constant. In the inequality above, all the operators act on both time and space directions. One can now follow the standard line of argument: either the quantity 
\((n + 1 + \tilde{\De} \phi) + (B + 1) \phi\) attains its maximum in the interior, in which case we have an upper bound on \(\sum_j {1 \over 1 + \phi_{j\bj }}\) and thereby on \(\tilde{\De} \phi\), or its maximum is attained on the boundary.
Since we have already found a uniform \(L^\infty\) bound,
finding an estimate on the boundary for the laplacian establishes a uniform estimate.


It is essential to note that the constant \(C\) on the right hand side depends only on the curvature properties of the underlying manifold \(M \time \Si\).
In fact, what is needed is a
lower bound of the bisectional curvature, \(\inf_{\calG_T, \al, \be} R_{\al \bal \be \bbe}\), which is, since \(M\) is asymptotically locally euclidean and \(\Si\) is flat,
bounded independent of \(T\).

In order to prove the boundedness of the quantity \(\tilde{\De} \phi_T\) at the boundary points we follow Chen's  approach  in \cite{Ch00}, 
which in turn was inspired by a previous work of B. Guan on the Dirichlet problem for the complex Monge-Amp\`ere equation \cite{Gu98}. 
Thanks to the behaviour of the boundary conditions,
the boundary estimate for the laplacian remains valid independent of \(T\).
Hence, the laplacian estimates remain valid independent of \(T\). 

Further, we know that for a fixed function \(f\) on the strip,
boundedness of laplacian leads to strict ellipticity of the operator, which, in turn, combined
with a version of the Evans-Krylov theory adapted to the operators of the complex hessian, one obtains \(C^{2,\ga}\) bounds, see \textsection 4 \cite{Siu87}.
Observe that the exponent \(\ga\) and the norm \(\Vert \phi \Vert_{C^{2,\ga}(\calG_T)} \) only depend on the \(L^\infty\) and the laplacian estimates, and are, therefore, uniform for all \(T\).
This finishes the proof of the existence of classical solutions for the boundary value problem 
\ref{ma-f-1} 
on the strip, with globally bounded laplacian.

\end{proof}


Let us turn to proving the \(L^\infty\) bounds. The upper bounds are obtained in this case as in 
\ref{l-infty-hom}.
If the data is not necessarily invariant in the \(S^1\) direction,
one can notice that the function is indeed sub-harmonic in the time direction and its maximum therefore appears on the boundary.
For the lower bound, consider the function:

\begin{equation}
\label{l-infty}
	\tilde{\psi} = C t ^ 2 + t \phi_0 + (1-t) \phi_1
\end{equation}

One may verify that this function is a lower barrier for the solutions once the constant \(C\) is chosen to be appropriately large.

As for the laplacian bounds, we observe that the arguments for proving the laplacian estimates based on maximum principle are only depend on a lower bound of the curvature of the reference metric \(\om\), 
which on an ALE space are bounded, and on the \(L^\infty\) estimates.
In particular, as the domains expand, the estimates are not affected once we can prove uniform estimates on the boundary.
 Since this is very similar to 

\subsection{Higher regularity of \(\eps\)-solutions}
As we shall see later in the section on the uniqueness of metrics of constant scalar curvature, 
for any positive \(\eps\) on the right hand side, 
we shall need suitable asymptotics for the curvature that will allow us to integrate by parts the terms that involve the higher derivatives of  the \(\eps\)-\emph{approximate geodesics}.
We therefore make the following assertion concerning  space derivatives of solutions.

\begin{proposition}
\label{higher-derivatives}
In the boundary value problem 
\eqref{ma-f-1},
where \(f>0\) is equal to \(\eps\) outside of a compact set,
for any \(\eps>0\) we have:

\begin{eqnarray}
\label{3-3}
	\vert \nabla_{\al} \nabla_\be \nabla_ \ga \phi \vert &\leq& C r^{-2m - 1}\\ \nonumber
	\vert \nabla_{\al} \nabla_\be \nabla_ \ga \nabla_\tta \phi \vert &\leq& C r^{-2m-2}
	  \textrm{  } \al, \be, \ga, \tta \in \{1,..., m, \bar{1}, ... \bar{m} \} \nonumber
\end{eqnarray}

wherein \(C\) depends on \(\eps > 0\).
As a consequence, 

\begin{eqnarray}
\label{3-4}
	\vert Rm(\om_\phi) \vert, \vert Rc(\om_\phi) \vert, \vert K(\om_\phi) \vert \leq C r^{- 2m-2 } 
\end{eqnarray}

wherein \(Rm(\om_\phi), Rc(\om_\phi)\) and \(K(\om_\phi)\) are the curvature tensor, the Ricci tensor, and the scalar curvature.
About the gradient of the volume form ration we have:

\begin{equation}
	\label{coro-higher-derivatives}
	\left \vert \nabla \left({\om_\phi^n \over \om^n} \right) \right \vert \leq C r^{-2m -1 } 
\end{equation}

Further, we have that 
\begin{eqnarray}
\label{3-5}
\vert \nabla _\al \nabla_\be \phi' \vert \leq C r^{-2m +1}
\end{eqnarray}

\end{proposition}

\begin{proof}
The proof is an application of Schuader estimates to the space derivatives and a type of boot-strap argument.
We can apply the Schauder estimates since we already know membership in \(C^{2,\al}\) of the potential for the \(\eps\)-solution.

We shall first derive estimates for the first space derivative.
Since at sufficiently far points  the right hand side is constant, 
at those points we obtain the following by differentiating the equation

\begin{eqnarray}
\label{de-phi-xi}
	\De_\phi \nabla_\xi \phi = 0
\end{eqnarray}

for some unit spatial direction \(\xi\).
We shall now use the fact that the quantity \(\nabla_\xi \phi\) using an appropriate barrier.
Since the boundary conditions for \(\nabla_\xi \phi\) decays like \(r^{-2m + 1}\),
one can easily verify that a function
 \(v:= r^{-\ka}(t - t^2) + r^{-2m + 1}\) is an upper barrier when \(\ka < 2m + 1 \). This in particular means that any space derivative decays at least at the rate of \(r^{-2m + 1}\), which is the same as the decay rate of the boundary conditions.

Using the decay rate for the first derivatives, we now prove some decay estimate for the space derivatives up to the third order along with their H\"older semi-norms.
Consider now the domains 
\(\Om_R \subset \Om'_R \subset M \times \Si\) 
defined as follows. 
Define 
\(\Om_R := \{ (x,t)| R - 1 < \rho(r) < R + 1 \} \), and \(\Om'_R := \{ (x,t)| R-2 < \rho(x) < R+2 \} \).
The decay rate we have obtained guarantees that on the pieces of the boundary \(\{\rho(x) = R-2\} \) and \(\{ \rho(x) = R+2 \}\), the quantity \(\nabla_\xi \phi\) is bounded by \(C R^{-2m + 1}\). Owing to the decay rates of the boundary conditions on the other hand, on the two components \(T_{0,1}:=\del (\Om'_R) \cap \{t=0,1 \}\) we have \(\Vert \nabla_\xi \phi \Vert_{C^{2,\al}(T_{0,1})} \leq CR^{-2m + 1} \).
We conclude, by the  Schauder estimates, that we have on \(\Om'_R\):
\[
	\Vert \nabla_\xi \phi \Vert_{C^{2,\al}(\Om_R)} \leq CR^{-2m + 1}
\]

As a result, any of the third order derivatives that have at least one spatial direction belong to \(C^\al\) and their \(C^\al\) norm is bounded by \(CR^{-2m + 1}\). 
Particularly, this proves \eqref{3-5}.

We can now differentiate \eqref{de-phi-xi} in a unit spatial direction 
\(\xi'\) to obtain:

\begin{equation}
\label{fourth-derivative}
\De_\phi \left( \nabla _{\xi'} \nabla_ \xi \phi \right) 
=
 g^{i  \bar{n}}_\phi g^{m \bj}_\phi \phi_{;m \bar{n} \xi'} \phi_{;\xi i \bj}
\end{equation}
wherein the Latin indices vary over both time and space coordinates.
Thanks to the decay estimates for the third derivatives  of the potential with at least one spatial direction, we observe that on domains \(\Om_R\) and \(\Om'_R\),
the right hand side of \eqref{fourth-derivative} satisfies:

\[
\Vert  g^{\al \bn}_\phi g^{\m \bbe}_\phi \phi_{\m \n \xi'} \phi_{\xi \al \bbe} \Vert_{C^\al(\Om'_R)} \leq CR^{-4m + 2}
\]

By using an argument similar to the poof of the \(C^{3,\al}\) estimate in the space direction, 
we have that 

\[
\Vert \phi \Vert_{C^{4,\al}} \leq C R^{-2m - 2}
\]

which finishes the proofs of \eqref{3-3} and \eqref{3-4}.

\end{proof}

\section{\emph{A priori} estimates for the geodesic equation}
\label{section-4}

As we have seen, the laplacian estimates are reduced to deriving the laplcian estimates at the boundary.
Further, the estimates of in \cite{Gu98} work in the case of degenerate right hand side as well and this was further used in \cite{Ch00}. 
The proof of Proposition 
\ref{l-infty-f-1}
therefore carry over to the case of degenerate right hand side.
We only prove the weighted \(L^\infty\) estimates.

\subsection{\(L^\infty\)-estimates}
\label{l-infty-hom}

Similar to the case of compact manifolds, we derive the \(L^\infty\) estimates. 
It will be enough to find a sub- and a super-solution in order to find upper and lower bounds on the function. \\
For the upper bound, notice that the function \(\phi\) is indeed convex in the time direction, namely \(\phi'' \geq 0\). 
Therefore, the upper bounds can only occur on the boundary points.
This, in particular, proves that the upper bound decays at the same rate as the boundary conditions.
As for the sub-solution,
one may consider any \(\Om\)-pluri-subharmonic function that restricts to the the boundary conditions. Consider the following function:
\[
	\psi(\fz, t) := t\phi_0 + (1-t) \phi_1 + C r^{3-4n} t^2
\]

wherein \(\ga\) is the exponent appearing in Theorem
\ref{ma-f-1}.
We show that if the constant \(C\) is chosen to be large enough, then one observes that \(\psi\) is indeed a sub-solution for the homogeneous problem.

To see this, we first see by direct calculation that 
\begin{eqnarray}
	\psi(\fz ,t)_{t\al} = O(r^{-2n + 1}) \\ \nonumber
	\psi(\fz ,t)_{t\bbe} = O(r^{-2n + 1}) \\ \nonumber
	\psi(\fz, t)_{tt} = O(r^{4n - 3}) \nonumber
\end{eqnarray}
 Using these rates of decay, by substituting these terms into the operator
 \( \left( \psi_{tt} - g^{\al \bbe}_\psi \psi_{t \al} \psi_{t \bbe} \right) {\om_\psi^n \over \om^n} \)
 we see that for sufficiently large \(C\), this expression is positive, whereby
 we conclude that \(\psi\) is a subsolution.



\section{\(\calK\)-energy, metrics of constant scalar curvature}
\label{mabuchi}
Although we shall not explicitly make use of it, one observes that
thanks to the asymptotic behaviour of the potentials and their derivatives we can define the Mabuchi \(L^2\)-metric on the space of potentials when \(n>2\).
What we need is will be the geodesics, the \(\eps\)-geodesics to be more precise, regardless of their relevance as
the extrema of the length functional.

In our calculation in the proof of the next theorem we shall need the following second order operator also sometimes referred to as the \emph{Lichnerowicz operator}.
It is the complementary part of the -real- hessian to the tensor \(dd^c u\). 
The operator \(\scrD\) 
applied to a real valued function \(u\), is defined in local coordinates as:

\begin{eqnarray}
	\scrD u:= \nabla_\al \nabla_\be u d\fz^\al \otimes d\fz^\be \nonumber
\end{eqnarray}

One important property of this operator that we shall use is that the if \(u\) lies in the kernel of \(\scrD\), then the vector field \[\uparrow \bar{\partial} u:= g^{\al \bbe}  {\del u \over \del \fz^{\bbe}} {\del \over \del \fz^\al}  \]
is holomorphic, cf. \textsection 1.22 in \cite{Gaud}.

In this section we shall always assume that the first Chern class, \(c_1(M)\), is non-positive.
As in the case of compact manifolds, we define the 
\(\calK\)-energy by its differential as follows:
  
\begin{equation}
\de_\psi \calK = - \int_M K_\phi  \psi \om_\phi^n  
\end{equation}

wherein \(K_\phi\) is the scalar curvature of the metric \(\om_\phi\).
Notice that by the asymptotics behaviour of the potentials, we know that the integral above is convergent.
In this section, we shall extend the proof given in \cite{Ch00} for the uniqueness of metrics of constant scalar curvature in each cohomology class of compact K\"ahler manifolds to the case ALE K\"ahler manifolds. It may be seen as a generalisation of the uniqueness theorem for ALE Ricci-flat K\"ahler metrics proved in \cite{Jo00}. 

We now turn our attention to the proof of the uniqueness assertion.
As we shall see, geodesics are not the only -owing to their lack of regularity, if they are at all- suitable curves for our purpose.
The geodesics nevertheless help us find the right curves for our purpose. 
As we shall see any curve in \(\calH\) with appropriately assigned, and not necessarily vanishing, acceleration, called \(\eps\)-geodesics in \cite{Ch00}, will do.

\begin{proof}[Proof of \ref{uniqueness}]

We follow Chen's  approach in \cite{Ch00}. 
Let us first assume that we have sufficiently regular geodesics
as was  done in \cite{Do96} to motivate our choice of curves. Evidently, a metric of constant zero scalar curvature is a stationary point of the functional \(\calK\).
Consider now two potentials \(\phi_0\) and \(\phi_1\) that realise two distinct metrics of constant scalar curvature, in particular, two stationary points of the \(\calK\)-energy.
By a formal calculation, one obtains
\[
{d^2 \calK \over dt^2} = \int_M \vert \scrD \phi'(t) \vert_\phi^2 \om_\phi^n
\]
But since we do not have higher regularity of solutions to the geodesic equation, 
the term \(\scrD \phi'\), which requires three bounded derivatives, cannot be defined. 
If the calculation were valid however, one could easily deduce that \(\calK\) had to be constant along the geodesic connecting \(\phi_0\) and \(\phi_1\).
Further, one deduces that the term \(\scrD \phi'\) would have to vanish.
As a result, \(\uparrow \bdel \phi'\) would have  to be the real part of a holomorphic vector field (cf. Lemma 1.22.2 in \cite{Gaud}).
Assuming that the same decay rates  proved in 
Proposition \ref{higher-derivatives}
hold for the solutions of the homogeneous complex Monge-Amp\`ere equation,
 \(\uparrow \bdel \phi'\) would have to be a  holomorphic vector field on the entire manifold \(M\) which decays at infinity.
 However, as we shall see in 
 Lemma \ref{bounded-v-field},
 there are no such vector fields but the trivial one. 
 Namely, the differential of the function \(\phi'\) would identically vanish .
One hence concludes that \(\phi'\) is constant in the space direction.
But since the only space-independent solution of the geodesic equation is the linear interpolation in time of the boundary conditions,
one has that \(\om_{\phi_0} = \om_{\phi_1}\).

To overcome the problem of lack of higher regularity of the solutions of the geodesic  problem we will use a family of geodesics that approximate the homogeneous problem and follow a similar path of reasoning for proving that \(\phi_0 = \phi_1\).

Let \(\om\) be the an ALE K\"ahler metric cohomologous to \(\om_0\) such that
\(\rho(\om) \leq 0\), wherein \(\rho(\om)\) is the Ricci form of the K\"ahler form \(\om\).
The existence of such an ALE metric \(\om \in [\om_0]\) is guaranteed since the first Chern class, \(c_1\), is assumed to be non-positive, and we further know that by the extension of the Calabi conjecture to the ALE space any closed real (1,1)-form \(\chi \in [\rho(\om_0)]\) with appropriate asymptotic behaviour
may be realised as the Ricci form of some unique ALE K\"ahler metric.  (see \textsection 8.4 and 8.5 of \cite{Jo00} for more details on the de Rham cohomology on ALE spaces and the proof of the Calabi conjecture on ALE spaces). 
In the case of vanishing first Chern class one could choose \(\om_0\) to be Ricci-flat, and in the case of negative \(c_1\) the form \(\om_0\) could be chosen to be
so that the Ricci form \(\rho(\om_0)\) is a negative on some bounded set and zero  outside of it.
Let us define \(\calG(\phi) :=  \phi'' - {1 \over 2}\vert d \phi' \vert_\phi^2\).
We now integrate by parts and obtain about the second derivative of the \(\calK\) along an arbitrary curve, in particular an approximate geodesic, \(\phi(t)\):
\begin{eqnarray}
\label{12}
	{d^2 \calK \over dt^2} 
	=
	\int_M \vert \scrD  \phi' \vert_\phi^2 \om_\phi^n
	-
	\int_M \calG(\phi) K_\phi \om_\phi^n
\end{eqnarray}
We now prove that the second integral is non-negative as well.
\begin{eqnarray}
-\int_M  \calG(\phi) K_\phi \om_\phi^n
&=& 
- \int_M \calG(\phi) Ric(\om_\phi) \wge \om_\phi^{n-1}
 \\ \nonumber
&=&
-\int_M \calG(\phi) \left(Ric(\om_\phi) - Ric(\om) \right) \wge \om_\phi^{n-1}  
-
\int_M \calG(\phi) Ric(\om) \wge \om_\phi^{n-1}
\\ \nonumber 
&=&
\int_M \calG(\phi) dd^c \log{\om^n_\phi \over \om^n} \wge \om_\phi^{n-1}
-
\int_M \calG(\phi) Ric(\om) \wge \om_\phi^{n-1}
 \\ \nonumber
&=&
-  \int_M d \calG(\phi) \wge d^c   \log{\om^n_\phi \over \om^n}  \wge \om_\phi^{n-1} 
-
\int_M \calG(\phi) Ric(\om) \wge \om_\phi^{n-1}
\\ \nonumber
&=&
\int d \calG(\phi) \wge d^c  \log \calG(\phi) \wge \om_\phi^{n-1}
-
\int_M \calG(\phi) Ric(\om) \wge \om_\phi^{n-1} \\ \nonumber
&=&
\int_M {\left \vert \nabla \calG(\phi) \right \vert_\phi^2 \over \calG(\phi)} \om_\phi^n
-
\int_M \calG(\phi) Ric(\om) \wge \om_\phi^{n-1} \nonumber
\end{eqnarray}

where we have used the fact that \(\log \calG(\phi) = - \log {\om_\phi^n \over \om^n} + \log \eps \). 
Notice that the integration by parts carried out above is meaningful by virtue of
the asymptotics proved in \ref{higher-derivatives} and \ref{coro-higher-derivatives}.
More precisely, the asymptotics guarantee the membership in \(L^1\) of
the integrands, which, then, by Gaffney's extension of  Stokes's theorem on compact riemannian manifold to the case of arbitrary complete manifolds \cite{Gaff54},
proves the validity of integration by parts.

Now since \(Ric(\om) \leq 0\) and \(\calG > 0\), the second term is a non-negative finite quantity.

In other words, along the \(\eps\)-approximate geodesic we have that the \(\calK\)-energy is convex.
Further, since the end points are scalar-flat metrics, 
they are stationary points of the \(\calK\)-energy.
Hence, 
the \(\calK\)-energy is constant along the path, and in particular \(\scrD \phi' = 0\) and we can repeat the argument given in the beginning of the proof with the assumption of the smoothness of geodesics to obtain \(\phi_0 = \phi_1\).

\end{proof}

We now state and prove the following lemma, similar to Lemma 4.3 in \cite{Ch-He-09}, which we used in the previous proof to guarantee that the term \(\phi'\) is indeed constant in space. 
\begin{lemma}
\label{bounded-v-field}
Let \(M\) be an ALE K\"ahler space and \(X\) a holomorphic vector field on \(M\).
Suppose that \(X\) decays to zero at infinity.
Then \(X\) must be the trivial vector field.
\end{lemma}

\begin{proof}[Proof of the lemma]
We know that an ALE manifold has a trivialisation outside of some compact set: \(M - K \simeq \left( \bbC^n - B(0,R) \right)/\Ga \) for some ball \(B(0,R)\) and some finite subgroup \(\Ga \subset U(n)\).
We can therefore lift the vector field \(X\) on \(M-K\) to a holomorphic vector field \(\tilde{X}\) on \(\bbC^n - B(0,R)\).
By Hartogs' theorem applied to the individual components of \(\tilde{X}\) however, one can extend the vector field to the entire complex space \(\bbC^n\) to obtain a global holomorphic vector field on \(\bbC^n\).
The decay estimate \eqref{3-5} along with the maximum principle imply that
each component of the vector field \(\tilde{X}\) vanished, hence so does \(X\).

\end{proof}

Having proved the uniqueness of metrics of constant scalar curvature in each K\"ahler class, 
we can now conclude this section by the proof of the boundedness from below of the \(\calK\)-energy on such K\"ahler manifolds.

\begin{proof} [Proof of \ref{boundedness}]
Let \(\psi\) be an arbitrary ALE potential cohomologous to \(\chi\), where \(\chi\) is defined as the proof of \ref{uniqueness}.
Let \(\Phi(t)\) be some smooth enough path connecting the two potentials.
By the calculations in the the proof of \ref{uniqueness}, we have that along the curve \(\Phi(t)\) the \(\calK\)-energy is convex. 
Also, since \(\chi\) realises the minimum of the \(\calK\)-energy, the first derivative of \(\calK\) along \(\Phi\) vanishes at \(\chi\).
Hence, \(\calK\) is strictly increasing along \(\Phi(t)\) which proves the claim.

\end{proof}

\begin{corollary}
\label{coro-scalar-flat}
Let \((M,\om)\) be an ALE K\"ahler manifold with \(c_1(M)=0\).
Then, \(Ric(\om)=0\) if and only if it is of constant zero scalar curvature.
In other words, any scalar-flat K\"ahler metric in this case is Ricci-flat.
Further, in each K\"ahler class there exists one and only one metric of constant scalar curvature which is further Ricci-flat.
\end{corollary}

\begin{proof}
Obviously, the fact that being Ricci-flat implies being scalar flat requires no proof. 
So we only prove the converse.
Uniqueness of the scalar-flat metrics in each K\"ahler class is given by Theorem \ref{uniqueness}.
By the work of Joyce, Theorem 8.5.1 in \cite{Jo00}, there always exists a Ricci-flat metric in each K\"ahler class when \(c_1 =0\) in the sense of Definition 
\ref{c-1-bundle}.
These two results together prove the claims.

\end{proof}

In the particular case of \(c_1=0\) we shall give a more direct proof of the uniqueness of ALE metrics of constant scalar curvature.
I am not aware of this proof having been adapted to the case of ALE manifolds and I found it worth mentioning here.

\begin{proposition}
\label{harmonic}
Let \((M, \om)\) be an ALE K\"ahler manifold. 
Then, the Ricci form \(\rho(\om)\) is co-closed, and therefore harmonic, if and only if the scalar curvature \(s\) is a constant. 
\end{proposition}

\begin{proof} [Proof of Proposition \ref{harmonic}]
The fact that the Ricci form, \(\rho\), is co-closed, and hence harmonic, is a punctual fact and independent of the global geometry of the K\"ahler space, cf. Proposition 1.18.2 in \cite{Gaud}.
It suffices now to prove that any harmonic form with appropriate asymptotics is indeed co-closed.

\end{proof}

We  can now give the following proof of the uniqueness of Corollary \ref{coro-scalar-flat} using the more classical approach.
\begin{proof}[Alternative proof of Corollary \ref{coro-scalar-flat}]
	By Proposition \ref{harmonic} we know that since the scalar curvature vanishes identically, the Ricci form \(\rho\) must be harmonic.
We may now evoke the Hodge-de Rham-Kodaira decomposition on ALE manifolds, see Theorem 8.4.1 in \cite{Jo00}. 
In particular, this means when \(c_1 = 0\), the only harmonic form is the trivial one.
Noting however that \(\rho \in {1 \over 2\pi} c_1\) 
yields \(\rho = 0\).
By the extension of the Calabi conjecture to ALE K\"ahler spaces detailed in 
\cite{Jo00}, 
we know that in the same K\"ahler class one there exists a Ricci-flat, and hence scalar-flat, metric.

\end{proof}



\end{document}